\documentclass[11 pt]{article}

\usepackage{amsfonts}
\usepackage{amsmath}
\usepackage{amsthm}
\usepackage{amssymb}
\usepackage{geometry}                
\usepackage{graphicx}
\usepackage{xcolor}

\newtheorem{lemma}{Lemma}
\newtheorem{theorem}{Theorem}
\newtheorem{proposition}{Proposition}
\newtheorem{corollary}{Corollary}

\newcommand{\z}{\mathbb{Z}}
\newcommand{\aw}{\operatorname{aw}}
\newcommand{\awu}{\operatorname{aw}_u}

\title{Rainbow Arithmetic Progressions in Finite Abelian Groups.}
\author{Michael Young\thanks{Iowa State University, \texttt{myoung@iastate.edu}}
}

\begin{document}

\maketitle
\abstract{

For positive integers $n$ and $k$, the \emph{anti-van der Waerden number} of $\z_n$, denoted by $\aw(\z_n,k)$, is the minimum number of colors needed to color the elements of the cyclic group of order $n$ and guarantee there is a rainbow arithmetic progression of length $k$. Butler et al. showed a reduction formula for $\aw(\z_{n},3) = 3$ in terms of the prime divisors of $n$.  In this paper, we analagously define the anti-van der Waerden number of a finite abelian group $G$ and show $\aw(G,3)$ is determined by the order of $G$ and the number of groups with even order in a direct sum isomorphic to $G$. The \emph{unitary anti-van der Waerden number} of a group is also defined and determined.

}

\vspace{.2in}


\section{Introduction}
Let $G$ be a finite additive abelian group. A \emph{$k$-term arithmetic progression} ($k$-AP) of $G$ is a sequence of the form $$a, a+d, a + 2d, \ldots, a+(k-1)d,$$ where $a,d \in G$. For the purposes of this paper, an arithmetic progression is referred to as a set of the form $\{a, a+d, a + 2d, \ldots, a+(k-1)d \}$. A $k$-AP is \emph{non-degenerate} if the arithmetic progression contains $k$ distinct elements; otherwise, the arithmetic progression is \emph{degenerate}.

An \emph{$r$-coloring} of  $G$ is a function $c: G \rightarrow [r]$, where $[r]:=\{1,\dots,r\}$. An $r$-coloring is \emph{exact} if $c$ is surjective. Given $c: G \rightarrow [r]$, an arithmetic progression is called {\em rainbow} (under $c$) if $c(a + id) \neq c(a+jd)$ for all $0 \le i < j \le k-1$. Given $P \subseteq G$, $c(P)$ denotes the set of colors assigned to the elements of $P$, i.e. $c(P) = \{ c(i): i \in P \}$.

The \emph{anti-van der Waerden number} $\aw(G,k)$ is the smallest $r$ such that every exact $r$-coloring of $G$ contains a rainbow $k$-term arithmetic progression. If $G$ has contains no $k$-AP, then $\aw(G,k)=|G|+1$ to be consistent with the property that there is a coloring with $\aw(G,k)-1$ colors that has no rainbow $k$-AP.

Throughout the paper, $\z_n$ will denoted the cyclic group of order $n$ consisting of the set $\{0, 1, \ldots, n-1\}$ under the operation of addition modulo $n$. Define $[\z_n]^s := \underbrace{\z_n \times \z_n \times \cdots \times \z_n}_{s \mbox{ times}}$.

Jungi\'c, Licht, Mahdian, Ne\u{s}etril, and Radoi\u{c}i\'c established several results on the existence of rainbow 3-APs in \cite{J}. Jungi\'c et. al. proved that every 3-coloring of $\mathbb{N}$, where each color class has density at least $1/6$, contains a rainbow 3-AP. They also prove results about rainbow 3-APs in $\z_n$. Other results on colorings of the integers with no rainbow 3-APs have been obtained in \cite{AF} and \cite{AM}.

Anti-van der Waerden numbers were first defined by Uherka in a preliminary study (see \cite{U13}). Butler et. al., in \cite{awpaper}, proved upper and lower bounds for anti-van der Waerden numbers of $[n]$ and $\z_n$ for $k$-APs, for $3 \le k$.

Many of the extremal colorings that are constructed to prove lower bounds of $\aw(G,3)$ require colorings that use some color exactly once, which leads to the need of the following definitions.

An $r$-coloring of $G$ is \emph{unitary} if there is an element of $G$ that is uniquely colored, which will be referred to as a \emph{unitary color}. (A unitary coloring is referred to as a singleton coloring in \cite{awpaper}.) The smallest $r$ such that every exact  $r$-coloring of $G$ that is unitary contains a rainbow $k$-term arithmetic progression is denoted by $\awu(G,k)$. Similar to the anti-van der Waerden number, $\awu(G,k)=|G|+1$ if $G$ has no $k$-AP.

Butler et. al. use Proposition \ref{prime3or4} to determine the exact value of $\aw(\z_n,3)$.
\begin{proposition}\label{prime3or4} \cite[Proposition 3.5]{awpaper}, 
For every prime number $p$, $$3\leq \awu(\z_p,3) = \aw(\z_p,3)\leq 4.$$
\end{proposition}

Let $n=2^{e_0}p_1^{e_1}p_2^{e_2}\cdots p_s^{e_s}$ such that 
$p_j$ is prime and $0 \leq e_j$ for $0 \leq j \leq s$, 
$\aw(\z_{p_j},3)=3$ for $ 1 \leq j \leq \ell$, and
$\aw(\z_{p_j},3)=4$ for $\ell + 1 \leq j \leq s$. Then corollary 3.15 in \cite{awpaper} can be stated as follows:

\begin{theorem}\label{awequals} \cite[Corollary 3.15]{awpaper}
For any integer $n\geq 2$,  
\[ 
\aw(\z_n,3) = \left\{
\begin{array}{cc} 
2 +\sum\limits_{j=1}^\ell e_j + \sum\limits_{j=\ell+1}^s 2e_j & \mbox{ if } e_0=0 \\
3+ \sum\limits_{j=1}^\ell e_j + \sum\limits_{j=\ell+1}^s 2e_j & \mbox{ if } 1 \le e_0
\end{array}
\right.
\]
\end{theorem}

In this paper, Theorem \ref{awequals} is extended to all finite abelian groups Theorem \ref{2m} is generalized to finite abelian groups with order that is a power of $2$.

\begin{theorem}\label{2m} \cite[Theorem 3.5]{J}
For all positive integers $m$, $$aw( \z_{2^m},3) =3.$$
\end{theorem}

In section \ref{sec:aw}, a closed formula for $\aw(G,3)$ is given. This closed formula is determined by the order of $G$ and the number of groups with even order in a direct sum isomorphic to $G$. In section \ref{sec:awu}, a similar closed formula for $\awu(G,3)$ is given.

\section{Anti-van der Waerden Numbers} \label{sec:aw}

In this section, a reduction formula for the anti-van der Waerden number of groups that have odd order is created. Determining the anti-van der Waerden number of an abelian group with odd order is equivalent to determining the anti-van der Waerden number of $\z_m \times \z_n$ for some  positive odd integers $m$ and $n$. First we provide a proof of a useful remark from \cite{awpaper}.

\begin{proposition}\cite[Remark 3.16]{J}
For all positive integers $n$, $$\awu(\z_n,3) = \aw(\z_n,3).$$
\end{proposition}

\begin{proof}
It is obvious that $\awu(\z_n,3) \le \aw(\z_n,3).$ The inequality $\aw(\z_n,3) \le \awu(\z_n,3)$ will be shown by induction on the number of odd prime divisors of $n$. It is obviously true if $n$ is a power of 2. Assume $n$ is not a power of 2.

Let $\z_n = G \times \z_p$, where $G$ be a finite cyclic group and $p$ be an odd prime. Let $c_G$ be a unitary coloring of $G$ with exactly $\awu(G,3)-1$ colors and no rainbow 3-AP, and $c_p$ be a unitary coloring of $\z_p$ with exactly $\awu(\z_p,3)-1$ different colors. Without loss of generality, let 0 be uniquely colored by $c_G$ and $c_p$. For each $(g,h) \in G \times \z_p$, define $c$ as follows:

\[
c(g,h) = \left\{
\begin{array}{cl}
c_G(g) & \mbox{ if } h = 0,\\
c_p(h) & \mbox{ if } h \neq 0\\
\end{array}
\right.
\]

Let $\{(a_1,a_2), (a_1 + d_1, a_2+d_2), (a_1 + 2d_1, a_2+2d_2) \}$ be a 3-AP of $G \times \z_p$. Since $p$ is odd, $\{ a_2, a_2+d_2, a_2+2d_2 \}$ is a non-degenerate 3-AP in $\z_p$. Therefore, $0$, $1$, or $3$ elements of $\{(a_1,a_2), (a_1 + d_1, a_2+d_2), (a_1 + 2d_1, a_2+2d_2) \}$ will be assigned a color by $c_G$.

If $a_2 = 0$ and  $d_2 = 0$, then the 3-AP is colored by $c_G$ and is not rainbow. If $a_2 \neq 0$ and  $d_2 = 0$, then all the elements of the 3-AP are colored with the same color. If $d_2 \neq 0$, then the 3-AP is colored by $c_p$ (since $c_p(0)$ is a unitary color) and is not rainbow. Therefore no 3-AP in $G \times \z_p$ is rainbow under $c$.

The color $c(0,0)$ is unique; therefore, $c$ is a unitary coloring of $G \times \z_p$. So,
\begin{eqnarray*}
\awu(G \times \z_p,3) -1 &\ge& |c(G \times \z_p)| \\
&=& \awu(G,3) + \awu(\z_p,3) -3\\
\mbox{(by induction hypothesis)} &=&\aw(G,3) + \aw(\z_p,3) -3  \\
 \mbox{(by Theorem \ref{awequals})} &=& \aw(G \times \z_p,3) -1.
\end{eqnarray*}


Therefore, $\awu(\z_n,3) \ge \aw(\z_n,3)$.

\end{proof}

Now a coloring with no rainbow 3-APs is constructed to determine a lower bound.

\begin{proposition} \label{aw_awu_lowerbound}
For all positive integers $n$, $$\aw(G,3) + \aw(\z_n,3)-2 \le aw(G \times \z_n, 3).$$
\end{proposition}

\begin{proof}
It suffices to show that $\aw(G,3) + \awu(\z_n,3)-2 \le aw(G \times \z_n, 3).$ For each $g \in G$, let $P_g = \{ (g, h) : h \in \z_n \}$. Let $c_{G}$ be a coloring of $G$ with $\aw(G,3)-1$ colors with no rainbow 3-AP and $c_n$ be a unitary coloring of $\z_n$ with $\awu(\z_n,3)-1$ colors with no rainbow 3-AP. Without loss of generality, assume that 0 is an element of $\z_n$ that is uniquely colored by $c_n$.

Now define a coloring of $G \times \z_n$ with $aw(G,3) + \awu(\z_n,3)-2$ colors as follows: 

\[
c(g,h) = \left\{
\begin{array}{cl}
c_G(g) & \mbox{ if } h \neq 0,\\
c_n(h) & \mbox{ if } h = 0.\\
\end{array}
\right.
\]

Under the coloring $c$, there can be no rainbow 3-AP in any $P_g$. Since $n$ is odd, every other 3-AP must contain an element from $P_a$, $P_{a+d}$, and $P_{a+2d}$ for some $a,d \in \z_n$. However, such a 3-AP is not rainbow because $\{a, a+d, a+2d\}$ is not a rainbow 3-AP under $c_G$.
\end{proof}

The main tool used for determining the anti-van der Waerden number of abelian groups with odd order is applying Lemma \ref{1_color_difference} to create a well-defined auxiliary coloring of a specific subgroup.

Let $G$ be a group and $n$ be an odd positive integer. Partition $G \times \z_n$ by letting  $P_g = \{ (g,x) | x \in \z_n \}$ for each $g \in G$. Without loss of generality, let $|c(P_g)| \le |c(P_0)|$ for all $g \in G$.

Since $n$ is odd, 2 has a unique multiplicative inverse in $\z_n$. Therefore, for every $x \in \z_n$ there exists a $d \in \z_n$ such that $x = 2d$.
So given an AP in $G$, say $\{x_1, y_1, z_1\}$, and $x_2, z_2 \in \z_n$, there exists a unique $y_2 \in \z_n$ such that $x_2 + z_2 = 2y_2$, which yields $\{(x_1, x_2), (y_1, y_2), (z_1, z_2) \}$, a 3-AP in $G \times \z_n$.

\begin{lemma} \label{1_color_difference}
If $c$ is a coloring of $G \times \z_n$ with no rainbow 3-AP, then $|c(P_g) \setminus c(P_0)| \le 1$ for all $g \in G$.
\end{lemma}

\begin{proof}
Assume there is a $g \in G$ such that $2 \le |c(P_g) \char92 c(P_0)|$. Let $\alpha, \beta \in c(P_g) \char92 c(P_0)$ and $\gamma, \rho \in c(P_0) \char92 c(P_g)$. By maximality of $c(P_0)$, $\gamma$ and $\rho$ exists, and neither are equal to $\alpha$ or $\beta$. 

If there exists a $z \in P_{2g}$ such that $c(z)$ is not a color in $c(P_0)$, then there is a $y \in P_{g}$ such that $c(y) \in \{\alpha, \beta \}$ and $c(y) \neq c(z)$. Therefore, there is an $x \in P_0$ such that $\{x, y, z \}$ is a 3-AP in $G \times \z_n$ and $x$ does not have the same color as $y$ or $z$. This is a contradiction since this arithmetic progression is rainbow.

If there exists a $z \in P_{2g}$ such that $c(z)$ is not a color in $c(P_g)$ and $g \neq |G|/2$, then there is an $x \in P_{0}$ such that $c(x) \in \{\gamma, \rho \}$ and $c(x) \neq c(z)$. Therefore, there is a $y \in P_0$ such that $\{x, y, z \}$ is a 3-AP in $G \times \z_n$ and $y$ does not have the same color as $x$ or $z$. This is a contradiction since this arithmetic progression is rainbow.

Therefore, $P_{2g}$ must contain every color in $c(P_0)$ and $c(P_g)$, which is a contradiction to the maximality of $c(P_0)$.
\end{proof}

If $|c(P_g) \setminus c(P_0)| \le 1$ for all $g \in G$, then the following auxiliary coloring of $G$ is well defined:

\[
\overline{c}(g) = \left\{
\begin{array}{lll}
\alpha && \mbox{if } c(P_g) \subset c(P_0),\\
c(P_g) \char92 c(P_0) && \mbox{otherwise. }\\
\end{array}
\right.
\]

The next lemma goes on to show that if $G \times \z_n$ does not contain a rainbow 3-AP, then $\overline{c}$ can not create a rainbow 3-AP in $G$.

\begin{lemma} \label{aux_coloring}
If $\overline{c}$ contains a rainbow 3-AP in $G$, then there exists a rainbow 3-AP in $G \times \z_n$.
\end{lemma}

\begin{proof}
Let $\{a, a+d, a+2d\}$ be a rainbow arithmetic progression colored by $\overline{c}$ in $G$. Without loss of generality, there are two cases to consider: $\overline{c}(a+d) \neq \alpha$ and $\overline{c}(a+d) = \alpha$.

If $\overline{c}(a+d) = \beta$ and $\overline{c}(a+2d) = \gamma$, then there exists an $x \in P_a$, $y \in P_{a+d}$, and $z \in P_{a + 2d}$ such that $\{ x, y, z \}$ is a 3-AP in $G \times \z_n$, $c(y) = \beta$, and $c(z) = \gamma$. However, $\overline{c}(a) \neq \beta, \gamma$, which implies $\beta, \gamma \notin P_a$, so $c(x) \neq \beta, \gamma$. This implies that $\{x,y,z\}$ is a rainbow arithmetic progression in $G \times \z_n$, which is a contradiction.

If $\overline{c}(a) = \beta$, $\overline{c}(a+d) = \alpha$, and $\overline{c}(a+2d) = \gamma$, then there exists an $x \in P_a$, $y \in P_{a+d}$, and $z \in P_{a + 2d}$ such that $\{ x, y, z \}$ is a 3-AP in $G \times \z_n$, $c(x) = \beta$ and $c(z) = \gamma$. However, $c(y) \neq \beta, \gamma$ because $\overline{c}(a+d) = \alpha$. This implies that $\{x,y,z\}$ is a rainbow arithmetic progression in $G \times \z_n$, which is a contradiction.
\end{proof}

\begin{theorem}\label{aw_G_Zp_equality}
If $G$ is a finite abelian group and $n$ is an odd positive integer, then $$\aw(G \times \z_n, 3) =  \aw(G,3) + \aw(\z_n,3) - 2.$$
\end{theorem}

\begin{proof}
The lower bound for $\aw(G \times \z_n, 3)$ is by Proposition \ref{aw_awu_lowerbound}. For the upper bound, it suffices to show $\aw(G \times \z_n, 3) \le  \aw(G,3) + \awu(\z_n,3) - 2.$ Let $c$ be a coloring of $G \times \z_n$ with $ \aw(G,3) + \awu(\z_n,3) - 2$ colors and no rainbow 3-APs. The statement will be proved by contradiction, showing that no such coloring exists.

For each $g \in G$, let $P_g = \{ (g, h): h \in \z_n \}$. Without loss of generality, let $|c(P_g)| \le |c(P_0)|$ for all $g \in G$. Since there are no rainbow 3-APs and $P_0$ is isomorphic to $\z_n$, $|c(P_0)| \le aw(\z_n)-1$. Also, by Lemma \ref{1_color_difference}, $|c(P_g) \char92 c(P_0)| \le 1$, for all $g \in G$. Define a coloring of $G$ as follows:

\[
\overline{c}(g) = \left\{
\begin{array}{lll}
\alpha && \mbox{if } c(P_g) \subset c(P_0),\\
c(P_g) \char92 c(P_0) && \mbox{otherwise. }\\
\end{array}
\right.
\]

The total number of colors used by $c$ is $|c(P_0)| + |\overline{c}(G)| - 1 \le (\aw(\z_n,3) -1) + (\aw(G,3)-1) -1$. Therefore $|c(P_0)| \le \aw(\z_n,3)$ or $|\overline{c}(G)| \le \aw(G,3)$

\end{proof}

This leads to the following corollary which implies that for positive odd integers $m$ and $n$, $\aw(\z_m \times \z_n,3) = \aw(\z_{mn},3)$.

\begin{corollary}\label{awformula}
Let $n$ be the largest odd divisor of the order of $G$. There exists a finite abelian group $G'$ such that the order of $G'$ is a power of $2$ and $$\aw(G,3) = \aw(G',3) + \aw(\z_n,3) -2.$$ 
\end{corollary}

\begin{proof}
For each odd prime $p$ and positive integer $e$, $\aw(\z_{p^e},3)= e(\aw(\z_p,3) -2) +2$, by Theorem \ref{awequals}. Theorem \ref{aw_G_Zp_equality} implies $\aw(\z_{p^{e_1}} \times \z_{p^{e_2}} \cdots \z_{p^{e_{\ell}}},3) =  2+ (\aw(\z_{p},3) - 2) \displaystyle\sum_{i=1}^{\ell} e_i.$ So the anti-van der Waerden number is the same for any two finite abelian groups having the same odd order.

Now let $G = G' \times \z_n$ and $n = \displaystyle\prod_{i=1}^{\ell} p_i^{e_i}$, where $p_i$ is an odd prime for all $i$, where $1 \le i \le \ell$. Then
\begin{eqnarray*}
\aw(G,3) &=& \aw(G' \times \z_n,3)\\
&=& \aw(G',3) + \displaystyle\sum_{i=1}^{\ell} e_i(\aw(\z_{p_i},3) - 2)\\ 
&=& \aw(G',3) + \aw(\z_n,3) -2.
\end{eqnarray*}
\end{proof}

\subsection{Groups with power of 2 order}
In order to completely use Corollary \ref{awformula} the anti-van der Waerden number of groups with order that is a power of 2 must be determined.

\begin{proposition}\label{times_2}
For any finite abelian group $G$, $$\aw(G \times \z_2, 3) \le 2\aw(G,3) - 1.$$
\end{proposition}

\begin{proof}
Let $A = \{ (g,0) : g \in G\}$ and $B = \{ (g,1) : g \in G \}.$ In any exact $( 2aw(G,3) - 1)$-coloring of $G \times \z_2$, either $A$ or $B$ will have at least $\aw(G,3)$ colors. Therefore, a rainbow 3-AP will exist since $A$ and $B$ are both isomorphic to $G$.
\end{proof}

An inductive argument, using Proposition \ref{times_2} as the base case, gives the following corollary.

\begin{corollary}\label{2_to_the_s}
For all positive integers $s$, $$\aw( [\z_2]^s, 3) \le 2^s + 1.$$
\end{corollary}

\begin{theorem}\label{2_to_the_m_s}
For $1 \le i \le s$, let $m_i$ be a positive integer. Then $$aw( \z_{2^{m_1}} \times \z_{2^{m_2}} \times \cdots \times \z_{2^{m_s}}, 3) = 2^s + 1.$$
\end{theorem}

\begin{proof}

Let $m_1 \le m_2 \le \ldots \le m_s$ and $x = (x_1, x_2, \ldots, x_s) \in  \z_{2^{m_1}} \times \z_{2^{m_2}} \times \cdots \times \z_{2^{m_s}}$. Define $c(x) = (x_1, x_2, \ldots, x_s) \mod 2$. The function $c$ is an exact $2^s$-coloring of $\z_{2^{m_1}} \times \z_{2^{m_2}} \times \cdots \times \z_{2^{m_s}}$. Since $c(x) = c(x+2d)$, for any $d \in  \z_{2^{m_1}} \times \z_{2^{m_2}} \times \cdots \times \z_{2^{m_s}}$, this coloring does not contain any rainbow arithmetic progressions. Therefore, $2^s + 1 \le aw( \z_{2^{m_1}} \times \z_{2^{m_2}} \times \cdots \times \z_{2^{m_s}}, 3).$

The proof of the upper bound is inductive on $(s,m_s)$. The base case of $(1,m)$ is true for all positive integers $m$ by Theorem \ref{2m} and the base case of $(s,1)$ is true for all positive integers $s$ by Corollary \ref{2_to_the_s}. Assume the statement is true for parameters $(s', m)$ for all $1 \le s' < s$ and $1 \le m < m_s$.

It will be shown that the statement is true for parameters $(s,m_s)$ by assuming there exists a coloring of $\z_{2^{m_1}} \times \z_{2^{m_2}} \times \cdots \times \z_{2^{m_s}}$ with exactly $2^s + 1$ colors and no rainbow 3-AP, then arriving at a contradiction.

For each $i \in \z_{2^{m_s}}$, let $P_i = \{ (x , i) : x \in \z_{2^{m_1}} \times \z_{2^{m_2}} \times \cdots \times \z_{2^{m_{(s-1)}}} \}$. So $P_i$ is isomorphic to $\z_{2^{m_1}} \times \z_{2^{m_2}} \times \cdots \times \z_{2^{m_{(s-1)}}}$ for all $i$. Let $A$ be the set of $P_i$ with $i$ even, and $B$ be the set of $P_i$ with $i$ odd. So $A$ and $B$ are both isomorphic to $\z_{2^{m_1}} \times \z_{2^{m_2}} \times \cdots \times \z_{2^{(m_s - 1)}}$. By the induction hypothesis, $A$ and $B$ both have at most $2^s$ colors. So there exists $\alpha \in c(A)\char92 c(B)$ and $\beta \in c(B) \char92 c(A)$.

Assume without loss of generality, $x_{\alpha} := c^{-1}(\alpha) \in P_0$ and $x_{\beta} := c^{-1}(\beta) \in P_j$, where $j$ is odd. Then $\{x_{\alpha}, x_{\beta}, 2x_{\beta} - x_{\alpha}\}$ is a 3-AP in $\z_{2^{m_1}} \times \z_{2^{m_2}} \times \cdots \times \z_{2^{m_s}}$. So $2x_{\beta} - x_{\alpha} \in P_{2j}$. Since there are no rainbow 3-APs $c( 2x_{\beta} - x_{\alpha} )$ must be $\alpha$. 

Similarly, $\{ (i-1)x_{\beta} - (i-2)x_{\alpha}, ix_{\beta} - (i-1)x_{\alpha}, (i+1)x_{\beta} - ix_{\alpha} \}$ is a 3-AP for all $i$ and $c((i-1)x_{\beta} - (i-2)x_{\alpha})$ must be equal to $c((i+1)x_{\beta} - ix_{\alpha})$ if it is not rainbow. This implies that $\alpha \in c(P_i)$ for all even $i$, and $\beta \in c(P_i)$, for all odd $i$.

By the induction hypothesis, $P_i$ has at most $2^{s-1}$ colors for all $i$. Therefore, $|c(P_0) \cup c(P_j)| \le 2^s$. So there exists a color $\gamma$ that is not in $c(P_0)$ or $c(P_j)$. Now define an exact 3-coloring of $\z_{2^{m_s}}$ as follows:

\[
\overline{c}(i) = \left\{
\begin{array}{lll}
\alpha && \mbox{if } \alpha \in c(P_i) \mbox{ and } \gamma \notin c(P_i),\\
\beta  && \mbox{if } \beta \in c(P_i) \mbox{ and } \gamma \notin c(P_i),\\
\gamma && \mbox{if } \gamma \in c(P_i).
\end{array}
\right.
\]

The coloring $\overline{c}$ is an exact 3-coloring and creates a rainbow 3-AP in $\z_{2^{m_s}}$ by Theorem \ref{2m}. Let $\{a, a+d, a+2d\}$ be such a rainbow arithmetic progression. Without loss of generality, there are two cases to consider: $\overline{c}(a+d) \neq \gamma$ and $\overline{c}(a+d) = \gamma$.

If $\overline{c}(a) = \alpha$, $\overline{c}(a+d) = \gamma$, and $\overline{c}(a+2d) = \beta$, then $a$ must be even and $a+2d$ must be odd, which is a contradiction.

If $\overline{c}(a) = \alpha$, $\overline{c}(a+d) = \beta$, and $\overline{c}(a+2d) = \gamma$, then there exists an $x \in P_a$, $y \in P_{a+d}$, and $z \in P_{a + 2d}$ such that $\{ x, y, z \}$ is a 3-AP in $\z_{2^{m_1}} \times \z_{2^{m_2}} \times \cdots \times \z_{2^{m_s}}$, $c(y) = \beta$ and $c(z) = \gamma$. However, $c(x) \neq \beta$ or $\gamma$ because $\overline{c}(a) = \alpha$. This implies that $\{x,y,z\}$ is a rainbow arithmetic progression in $\z_{2^{m_1}} \times \z_{2^{m_2}} \times \cdots \times \z_{2^{m_s}}$, which is a contradiction.

Therefore, $aw( \z_{2^{m_1}} \times \z_{2^{m_2}} \times \cdots \times \z_{2^{m_s}}, 3) \le 2^s + 1.$
\end{proof}

\section{Unitary anti-van der Waerden Numbers} \label{sec:awu}

\begin{proposition} \label{awu_lowerbound}
For all positive integers $p$ and $q$, $$\awu(\z_p,3) + \awu(\z_q,3)-2 \le \awu(\z_p \times \z_q, 3).$$
\end{proposition}

\begin{proof}
This is the same as the proof of Proposition \ref{aw_awu_lowerbound} with $\overline{c}_{\z_p}$ changed to a unitary coloring of $\z_p$ with $\awu(\z_p,3)-1$ colors and no rainbow 3-AP.
\end{proof}

\begin{theorem}\label{singleton_bound}
For any positive odd integers $n$, $$aw_u( G \times \z_{n}, 3) = aw_u(G,3) + aw_u(\z_n, 3) - 2.$$
\end{theorem}

\begin{proof}
The lower bound is a direct result of Proposition \ref{awu_lowerbound}. So it suffices to show the upper bound. Assume $c$ is a coloring of $G \times \z_n$ that is unitary with exactly $aw_u(G,3) + aw_u(\z_n, 3) - 2$ colors and no rainbow 3-AP. For all $h \in \z_n$, let $P_h = \{ (g,h) \mbox{ } | \mbox{ } g \in G \}$. Without loss of generality, let $|c(P_h)| \le |c(P_0)|$ for all $h \in \z_n$.

By Lemma \ref{1_color_difference}, $|c(P_h) \char92 c(P_0)| \le 1$, for all $h \in \z_n$. Define a coloring of $\z_n$ as follows:
\[
\overline{c}(g) = \left\{
\begin{array}{lll}
\alpha && \mbox{if } c(P_h) \subset c(P_0),\\
c(P_h) \char92 c(P_0) && \mbox{otherwise. }\\
\end{array}
\right.
\] Let $\rho$ be a color used exactly once by $c$ to color $G \times \z_n$. Now consider the two cases in which $\rho \in P_0$ and $\rho \notin P_0$.

\emph{Case 1:} If $\rho \in c(P_0)$, then $|c(P_0)| \le \awu(G)-1$. Therefore $\awu(\z_n,3)  = (\awu(G,3) + \awu(\z_n, 3) - 2) -  (\awu(G)-1) + 1 \le |\overline{c}(\z_n)|$. Since $\aw(\z_n,3) = \awu(\z_n,3)$, Lemma \ref{aux_coloring} implies that $c$ creates a rainbow 3-AP.

\emph{Case 2:} If $\rho \in c(P_d)$, where $0 \neq d$, then $\overline{c}$ must be a unitary coloring of $\z_n$ and not have any 3-APs by Lemma \ref{aux_coloring}. So $|\overline{c}(\z_n)| \le \awu(\z_n,3) - 1$, which implies $\awu(G,3) = (\awu(G,3) + \awu(\z_n, 3) - 2) -  (\awu(\z_n)-2) \le |c(P_0)|$. 

If there exists $\gamma \in c(P_0) \setminus c(P_{-d})$, then there is an $x \in G$ such that $c(x, 0) = \gamma$. Now choose $(y,d)$ such that $c(y,d) = \rho$. Then $\{(2x-y,-d), (x,0), (y, d) \}$ is a rainbow 3-AP. Therefore, $|c(P_0)| = |c(P_{-d})|$. If $|c(P_0)| > |c(P_{d})|$, then there exist $\beta, \gamma \in c(P_0) \setminus c(P_{d})$ because $\rho \notin c(P_0)$. Now a rainbow 3-AP can be attained by choosing elements of $P_0$ and $P_{-d}$ that are assigned $\beta$ and $\gamma$, respectively, and the corresponding element of $P_{d}$. Hence, $\awu(G,3) \le |c(P_0)| = |c(P_d)|$. However, since there is only one element in $P_d$ with the color $\rho$, this implies that $P_d$ contains a rainbow 3-AP, which is a contradiction.
\end{proof}

Theorem \ref{singleton_bound} yields the following corollary that is analogous to corollary \ref{awformula}.

\begin{corollary}\label{awuformula}
Let $n$ be the largest odd divisor of the order of $G$. There exists a finite abelian group $G'$ such that the order of $G'$ is a power of $2$ and $$\awu(G,3) = \awu(G',3) + \awu(\z_n,3) -2.$$ 
\end{corollary}

\subsection{Groups with power of 2 order}
\begin{proposition}\label{au_2_to_the_s}
For all positive integers $s$, $$\awu([\z_2]^s,3) = s+2.$$
\end{proposition}

\begin{proof}
This proof is by induction on $s$. The base case of $s=1$ is trivial since $2 < \awu([\z_2]^s,3) \le \aw([\z_2]^s,3) = 3.$ Assume $1 < s$.

Let $c'$ be a coloring of $[\z_2]^{s-1}$ with s+1 colors, no rainbow 3-AP and $a$ be an element of $[\z_2]^{s-1}$ that does not share a color with any other element. For all $g \in [\z_2]^{s-1}$ and $h \in \z_2$, except $(a,0)$, let $c(g,h) = c'(g)$ and assign $c(a,0)$ a new color. Then $c$ is a unitary $(s+1)$-coloring of $[\z_2]^{s}$ with no rainbow 3-AP. So, $s+2 \le \awu([\z_2]^s,3)$.

Now assume $[\z_2]^s$ is colored with a unitary coloring that has exactly $s+2$ colors and no rainbow 3-AP. Let $A = \{(g,0): g \in [\z_2]^{s-1} \}$ and $B = \{(g,1): g \in [\z_2]^{s-1} \}$. Without loss of generality, let $(a,0)$ be an element of $[\z_2]^s$ that does not share a color with any other element. Therefore, by induction, $|c(A)| \le s$. So there exists 2 colors, $\alpha, \beta \in c(B) \setminus c(A)$. Then there exists an element $(b,1)$ such that $c(b,1) \in \{\alpha, \beta\}$ and $2a \neq 2b$. Therefore, the 3-AP $\{(a,0), (b,1), (2b-a,0) \}$ must be rainbow, which is a contradiction. So, $\awu([\z_2]^s,3) \le s+2$.
\end{proof}

\begin{theorem}\label{awu_2_to_m_s}
For $1 \le i \le s$, let $m_i$ be a positive integer. Then $$\awu(\z_{2^{m_1}} \times \cdots \times \z_{2^{m_s}},3) = s+2.$$
\end{theorem}

\begin{proof}
This proof is inductive on $\displaystyle\sum_{i=1}^s m_i$. The base case of $\displaystyle\sum_{i=1}^s m_i = s$ is true by Proposition \ref{au_2_to_the_s}. So assume $s < \displaystyle\sum_{i=1}^s m_i$ and $2 \le m_s$.

Let $c'$ be a coloring of $\z_{2^{m_1}} \times \cdots \times \z_{2^{m_s}}$ with s+1 colors, no rainbow 3-AP and $a$ be an element of $\z_{2^{m_1}} \times \cdots \times \z_{2^{m_{(s-1)}}}$ that does not share a color with any other element. For all $g \in \z_{2^{m_1}} \times \cdots \times \z_{2^{m_{(s-1)}}}$ and $h \in \z_{2^{m_s}}$, except $(a,0)$, let $c(g,h) = c'(g)$ and assign $c(a,0)$ a new color. Then $c$ is a unitary $(s+1)$-coloring of $\z_{2^{m_1}} \times \cdots \times \z_{2^{m_s}}$ with no rainbow 3-AP. So, $s+2 \le \awu(\z_{2^{m_1}} \times \cdots \times \z_{2^{m_s}},3)$.

Now assume $\z_{2^{m_1}} \times \cdots \times \z_{2^{m_s}}$ is colored with a unitary coloring that has exactly $s+2$ colors and no rainbow 3-AP. Let $A= \{(g,h): g \in \z_{2^{m_1}} \times \cdots \times \z_{2^{m_{s-1}}}, h \in \z_{2^{m_s}}, \mbox{ and } h \mbox{ is even} \}$ and $B = \{(g,h): g \in \z_{2^{m_1}} \times \cdots \times \z_{2^{m_{s-1}}}, h \in \z_{2^{m_s}}, \mbox{ and } h \mbox{ is odd} \}$. Without loss of generality, let $(a,0)$ be an element of $\z_{2^{m_1}} \times \cdots \times \z_{2^{m_s}}$ that does not share a color with any other element. Therefore, by induction, $|c(A)| \le s$. So there exists 2 colors, $\alpha, \beta \in c(B) \setminus c(A)$. Then there exists an element $(b,2j+1)$ such that  $c(b,2j+1) \in \{\alpha, \beta\}$ and $2a \neq 2b$. Therefore, the 3-AP $\{(a,0), (b,2j+1), (2b-a,4j+2) \}$ must be rainbow, which is a contradiction. So, $\awu(\z_{2^{m_1}} \times \cdots \times \z_{2^{m_s}},3) \le s+2$.
\end{proof}

\begin{corollary}
Let $G$ be a finite abelian group. Then $\aw(G,3) = \awu(G,3)$ if and only if the order of $G$ is odd or $G$ is cyclic.
\end{corollary}

\begin{proof}
By Corollary \ref{awformula} and Theorem \ref{awu_2_to_m_s}, $$\aw(G,3) = 2^s + \aw(\z_n,3) - 1,$$ for some nonnegative integer $s$ and odd integer $n$. By Corollary \ref{awuformula} and Theorem \ref{2_to_the_m_s}, $$\awu(G,3) = s + \awu(\z_n,3),$$ for the same $s$ and $n$. Therefore, $\aw(G,3) = \awu(G,3)$ if and only if $2^s - 1 = s$; hence, $\aw(G,3) = \awu(G,3)$ if and only if $s$ is $0$ or $1$.
\end{proof}

\end{document}